\documentclass[12pt,reqno]{amsart}

\usepackage{amssymb}
\usepackage{amsmath}
\usepackage{latexsym}
\usepackage{amsthm}
\usepackage{epsfig}
\usepackage{color}
\usepackage{comment}
\usepackage{amsfonts,amssymb,mathrsfs,amscd}

\usepackage{enumitem}

\theoremstyle{plain}
\newtheorem{theorem}{Theorem}[section]
\newtheorem{proposition}{Proposition}[section]
\newtheorem{corollary}{Corollary}[section]
\newtheorem{lemma}{Lemma}[section]
\newtheorem{remark}{\bf Remark}[section]
\theoremstyle{definition}

\newcommand{\rot}{\mathop\mathrm{rot}}
\newcommand{\grad}{\mathop\mathrm{grad}}
\renewcommand{\div}{\mathop\mathrm{div}}

\title
[Berezin--Li--Yau inequalities on domains on the sphere]
{Berezin--Li--Yau inequalities on domains on the sphere}

\author[A.Ilyin and A.Laptev] {Alexei  Ilyin and Ari Laptev}

\begin{document}
 \begin{abstract}
We prove Berezin--Li--Yau inequalities for the
Dirichlet and Neumann eigenvalues on domains on the sphere
$\mathbb{S}^{d-1}$.  The case of $\mathbb{S}^{2}$ is treated
in greater detail, including the vector Dirichlet Laplacian and
the Stokes operator.

\end{abstract}


\address
{\noindent\newline  Keldysh Institute of Applied Mathematics;
\newline
Imperial College London and Institute Mittag--Leffler;}
\email{ ilyin@keldysh.ru;  a.laptev@imperial.ac.uk}

\maketitle

\setcounter{equation}{0}
\section{Introduction}
\label{sec0}

Let $\Omega$ be a bounded domain in $\mathbb{R}^d$ with volume $|\Omega|$. We denote by
$\{\lambda_k\}_{k=1}^\infty$ and $\{\mu_k\}_{k=1}^\infty$ the eigenvalues
of  the  Laplacian $-\Delta$  with Dirichlet and Neumann boundary conditions, respectively.
In the case of the Neumann boundary conditions we additionally assume that
boundary $\partial \Omega$ is sufficiently regular so that the embedding
$H^1(\Omega)\hookrightarrow L_2(\Omega)$ is compact.

The following bounds hold  for the Riesz means of  $\lambda_k$ and $\mu_k$
\begin{equation}\label{BL}
\aligned
\sum(\lambda-\lambda_k)_+^\sigma\le\mathrm{L}^\mathrm{cl}_{\sigma,d}|\Omega|\lambda^{\sigma+d/2},\\
\sum(\lambda-\mu_k)_+^\sigma\ge\mathrm{L}^\mathrm{cl}_{\sigma,d}|\Omega|\lambda^{\sigma+d/2},
\endaligned
\end{equation}
where $x_+=\max(0,x)$, $\sigma\ge1$,  and
\begin{equation}\label{class}
\mathrm{L}^\mathrm{cl}_{\sigma,d}=
\frac1{(2\pi)^d}\int_{\mathbb{R}^d}(1-|\xi|^2)^\sigma_+d\xi=
\frac{\Gamma(\sigma+1)}{(4\pi)^{2/d}\Gamma(\sigma+d/2+1)}.
\end{equation}
The first one follows from more general results of \cite{Berezin}.
A direct proof of both is given in \cite{LapFA}.

A lower bound for the sums of the Dirichlet eigenvalues was
obtained in~\cite{Li-Yau}
\begin{equation}\label{LiYau}
\sum_{k=1}^n\lambda_k\ge\frac d{2+d}\left(\frac{(2\pi)^d}{\omega_d|\Omega|}\right)^{2/d} n^{1+2/d},
\end{equation}
and a similar upper bound for the Neumann eigenvalues was proved in~\cite{Kroger}
\begin{equation}\label{Kroger}
\sum_{k=1}^n\mu_k\le\frac d{2+d}\left(\frac{(2\pi)^d}{\omega_d|\Omega|}\right)^{2/d} n^{1+2/d}.
\end{equation}
We  point out that the fact that the coefficients are constant
is essential in the proof.
Also, the constants in~\eqref{BL} and \eqref{LiYau},\eqref{Kroger} are
sharp in view of the Weyl asymptotic formula for the eigenvalues.
Finally, it was observed in~\cite{L-W} that inequalities \eqref{BL}
with $\sigma=1$   and
\eqref{LiYau}, \eqref{Kroger} are pairwise equivalent,
and the equivalence is realized by the Legendre transform and
explains why they are called Berezin--Li--Yau inequalities.

Note that for domains $\Omega\subset\mathbb{R}^d$ there are many results where the authors obtained Berezin--Li--Yau inequalities with remainder terms. Among them are  \cite{G-L-W}, \cite{W} (with $\sigma\ge3/2$ )  for the Berezin inequality and
\cite{F-L-U}, \cite{M}, \cite{K-V-W}, \cite{Y}, \cite{YY} for the Li-Yau estimate.
Moreover, improved such inequalities for magnetic operators were recently proved in \cite{KW}.

Our interest in this paper is in proving Berezin--Li--Yau inequalities
for the Dirichlet and Neumann eigenvalues of the Laplace--Beltrami operator  on a domain $\Omega$ on
the unit sphere $\mathbb{S}^{d-1}$. Now
 $\Omega$ is a (curved) domain on $\mathbb{S}^{d-1}$ with $(d-1)$-dimensional
surface area  $|\Omega|\le |\mathbb{S}^{d-1}|=:\sigma_d$. As before we denote by
$\{\lambda_k\}_{k=1}^\infty$ and $\{\mu_k\}_{k=1}^\infty$ the eigenvalues
of  the  Laplace--Beltrami operator $-\Delta$
with Dirichlet and Neumann boundary conditions, respectively.

The main general results of Section~\ref{sec1} are the following bounds for the Riesz means
of order $\sigma\ge1$
of the eigenvalues $\{\lambda_k\}_{k=1}^\infty$ and $\{\mu_k\}_{k=1}^\infty$:
\begin{equation}\label{Riesz}
\aligned
\sum_{j=1}^\infty(\lambda-\lambda_j)_+^\sigma
\le\frac{|\Omega|}{\sigma_d}\sum_{n=0}^\infty(\lambda-\Lambda_n)_+^\sigma k_d(n),\\
\sum_{j=1}^\infty(\lambda-\mu_j)_+^\sigma
\ge\frac{|\Omega|}{\sigma_d}\sum_{n=0}^\infty(\lambda-\Lambda_n)_+^\sigma k_d(n).
\endaligned
\end{equation}
Here $\Lambda_n$ and $k_d(n)$ are the eigenvalues and their multiplicities
of the Laplace operator on the whole sphere $\mathbb{S}^{d-1}$, see
\eqref{harmonics}--\eqref{mult}.

Next, we set $\sigma=1$  so that the right-hand side, denoted below
by $F_{\mathbb{S}^{d-1}}(\lambda)$,
is a continuous piecewise linear function with change of slope at $\lambda=\Lambda_N$, $N=0,1\dots$,
and by means of the explicit expression for $F_{\mathbb{S}^{d-1}}(\Lambda_N)$
proved in Section~\ref{secApp} we show in Section~\ref{sec4} that
in all dimensions
$$
F_{\mathbb{S}^{d-1}}(\lambda):=\frac{|\Omega|}{\sigma_d}\sum_{n=0}^\infty(\lambda-\Lambda_n)_+ k_d(n)\ge
\mathrm{L}^\mathrm{cl}_{1,d-1}|\Omega|\lambda^{1+(d-1)/2},
$$
where the inequality is strict for all $\lambda>0$ when $d-1\ge3$, while in 2D case
we have equality for $\lambda=\Lambda_N$. This gives that
$$
\sum_{j=1}^\infty(\lambda-\mu_j)_+\ge\mathrm{L}^\mathrm{cl}_{1,d-1}|\Omega|\lambda^{1+(d-1)/2}.
$$

Applying the Legendre transform for both sides we obtain a Li--Yau-type upper bound
for the sums of the first $n$  eigenvalues of the
Neumann Laplacian on $\Omega\subseteq\mathbb{S}^{d-1}$
$$
\sum_{k=1}^n\mu_k\le
\frac {d-1}{2+d-1}\left(\frac{(2\pi)^{d-1}}{\omega_{d-1}|\Omega|}\right)^{2/(d-1)}\,n^{1+2/(d-1)},
$$
which looks exactly the same as in the case of the Euclidean space.

The situation with the Dirichlet eigenvalues is different and a Li--Yau estimate in  the form
\eqref{LiYau} cannot hold, since the first eigenvalue on the whole
sphere is $0$. Therefore we restrict ourselves to the case of $\mathbb{S}^2$,
where we take advantage of the fact that
$F_{\mathbb{S}^{2}}(\Lambda_N)=\mathrm{L}^\mathrm{cl}_{1,2}|\Omega|\Lambda_N^{2}$,
and by evaluating the Legendre transform of $F_{\mathbb{S}^{2}}(\lambda)$
we  obtain a Li--Yau-type lower bound:
$$
\sum_{k=1}^n\lambda_k\ge\frac{2\pi}{|\Omega|}\,n\left(n-\frac{|\Omega|}{4\pi}\right),
$$
which turns into the equality when $\Omega=\mathbb{S}^2$ and
$n=N^2$, and which properly takes into account the behaviour of the
first eigenvalue as $|\Omega|\to |\mathbb{S}^2|=4\pi$:
\begin{equation}\label{lambda1intro}
\lambda_1=\lambda_1(\Omega)\ge\frac{2\pi}{|\Omega|}\,\left(1-\frac{|\Omega|}{4\pi}\right).
\end{equation}

In this connection we observe that sharp estimates for the first eigenvalue of
the Schr\"odinder operator on $\mathbb{S}^{d-1}$ without the exclusion of the zero mode
were obtained in \cite{D-E-L}.

The proof of the inequalities for Riesz means~\eqref{Riesz}, on which the subsequent analysis
is based, essentially relies on the pointwise  identity~\eqref{identity}
for the orthonormal spherical harmonics.

A similar pointwise identity holds for the gradients
of the spherical harmonics~\eqref{identity-vec} and makes it
possible to prove Berezin--Li--Yau inequalities  for the
Dirichlet eigenvalues of the vector Laplacian. Unlike the scalar case,
the vector Laplacian (we deal with the Laplace--de Rham operator)  is
strictly positive, since the sphere is simply connected. We consider
only the two dimensional case, where a divergence free
vector function have a scalar stream function, so that the
identity~\eqref{identity-vec} works equally well in the invariant
spaces of potential and divergence free vector functions. We obtain
a Li--Yau-type inequality for the vector Dirichlet Laplacian, which
looks exactly the same as in the case $\Omega\subset\mathbb{R}^2$. A more interesting
example is the Stokes operator on a domain $\Omega\subseteq\mathbb{S}^2$:
\begin{equation}\label{Stokessmooth}
\aligned
&-\mathbf{\Delta} u_j \,+\, \nabla p_j\,=\,\nu_ju_j,\\
&\div u_j\,=\,0,\,\,\,u_j\vert_{\partial\Omega }\,=\,0,
\endaligned
\end{equation}
where the scalars $p_j$ are the corresponding pressures. We show that
$$
\sum_{k=1}^n\nu_k\ge\frac{2\pi}{|\Omega|}\,n^2.
$$
We observe that this estimate is also  exactly the same as the
Li--Yau bound for the Stokes operator~\cite{Il_Stokes}, \cite{Il_Stokes_Melas}
 in the 2d case  $\Omega\subset\mathbb{R}^2$.

In conclusion we  recall the basic facts concerning the Laplace
operator on the sphere $\mathbb{S}^{d-1}$~(see, for instance, \cite{Mikhlin,S-W}).
The one dimensional case $d=2$ is somewhat special, so we assume below that
$d\ge3$. We have
for the (scalar) Laplace--Beltrami operator
$\Delta=\div\nabla$:
\begin{equation}\label{harmonics}
-\Delta Y_n^k=\Lambda_n Y_n^k,\quad
k=1,\dots,k_d(n),\quad n=0,1,2,\dots.
\end{equation}
Here the $Y_n^k$ are the orthonormal spherical harmonics which
are chosen to be real-valued.
Each eigenvalue $\Lambda_n=n(n+d-2)$
has multiplicity
\begin{equation}\label{mult}
k_d(n)=\frac1{d-2}\binom{n+d-3}{d-3}(2n+d-2).
\end{equation}
For example, for $d=3$ we have for $\mathbb{S}^2$
$\Lambda_n=n(n+1)$, $k_3(n)=2n+1$.

The following identity is essential in what follows~\cite{Mikhlin,S-W}:
for any $s\in\mathbb{S}^{d-1}$
\begin{equation}\label{identity}
\sum_{k=1}^{k_d(n)}Y_n^k(s)^2=\frac{k_d(n)}{\sigma_d},
\end{equation}
where $|S^{d-1}|=:\sigma_d=2\pi^{d/2}/\Gamma(d/2)$ is the surface area of $S^{d-1}$.
This identity, in turn, follows from the addition theorem
for the spherical harmonics
\begin{equation}\label{Gegen}
\sum_{k=1}^{k_d(n)}Y_n^k(s)Y_n^k(s_0)=\frac{2n+d-2}{(d-2)\sigma_d}
P_n^\lambda(s\cdot s_0),\ s\cdot s_0=\cos\gamma,
\end{equation}
where $\gamma$ is the angle between $s$ and $s_0$, $\lambda=(d-2)/2$,
and~$P_n^\lambda(t)$ are the Gegenbauer polynomials associated with~$\lambda$
which can be defined in terms of a generating function
$$
(1-2rt+r^2)^{-\lambda}=\sum_{n=1}^\infty P_n^\lambda(t)r^n,\quad|t|\le 1,\
|r|<1.
$$

For the two-dimensional sphere the Gegenbauer polynomials $P_n^{1/2}(t)$ are the classical
Legendre polynomials $P_n(t)$, and \eqref{Gegen} goes over to
the Laplace addition theorem for the spherical harmonics on $\mathbb{S}^2$:
\begin{equation}\label{Lapl-add}
\sum_{k=1}^{2n+1}Y_n^k(s)Y_n^k(s_0)=\frac{2n+1}{4\pi}\,
P_n(s\cdot s_0).
\end{equation}

In the
vector case we have the identity for the gradients
of spherical harmonics that is similar to~\eqref{identity} (see \cite{I93}): for any $s\in\mathbb{S}^{d-1}$
\begin{equation}\label{identity-vec}
\sum_{k=1}^{k_d(n)}|\nabla Y_n^k(s)|^2=\Lambda_n\frac{k_d(n)}{\sigma_d}.
\end{equation}

This identity is especially useful for inequalities for
vector functions on $\mathbb{S}^2$ and we prove it for the sake
of completeness. Substituting $\varphi(s)=Y_n^k(s)$ into
the identity
$$
\Delta\varphi^2=2\varphi\Delta\varphi+2|\nabla\varphi|^2
$$
we sum the results over $k=1,\dots,k_d(n)$. In view of \eqref{identity}
the left-hand side vanishes and we obtain~\eqref{identity-vec}
since the $Y_n^k(s)$'s are the eigenfunctions corresponding to $\Lambda_n$.

\setcounter{equation}{0}
\section{Inequalities for Riesz means of the  Dirichlet and Neumann
eigenvalues on domains on the sphere}\label{sec1}

Let $\Omega\subseteq\mathbb{S}^{d-1}$ be a (curved) domain on $\mathbb{S}^{d-1}$
with $(d-1)$-dimensional surface measure  $|\Omega|\le\sigma_d$.
We denote by $H^1(\Omega)$ the standard Sobolev space on $\Omega$, and by
$H^1_0(\Omega)$ its closed subspace of functions vanishing on $\partial \Omega$.
Next, we define the Dirichlet Laplacian $\Delta^D_\Omega$ and the  Neumann Laplacian
$\Delta^N_\Omega$
via the quadratic forms with domains $H^1_0(\Omega)$ and $H^1(\Omega)$,
respectively.

We assume (in the case of the Neumann Laplacian)  that
the boundary $\partial\Omega$ is sufficiently regular so that
the embedding $H^1(\Omega)\hookrightarrow L_2(\Omega)$ is compact.

To fix notation we write the Dirichlet and Neumann eigenvalue problems
\begin{equation}\label{DN}
\aligned
-\Delta \psi_j=\lambda_j\psi_j,\quad\psi_j\vert_{\partial\Omega}=0,\qquad
0\le\lambda_1\le\lambda_2\le\dots \to+\infty,\\
-\Delta \omega_j=\mu_j\omega_j,\quad
\frac{\partial\omega_j}{\partial n}{\vert_{\partial\Omega}}=0,\qquad
0=\mu_1\le\mu_2\le\dots \to+\infty,
\endaligned
\end{equation}
where $n$ is the unit normal to $\partial\Omega$ tangent to $\mathbb{S}^{d-1}$.
Here $\lambda_1>0$ if $\Omega$ is a proper domain on $\mathbb{S}^{d-1}$:
$\mathrm{meas}(\mathbb{S}^{d-1}\setminus\Omega)>0$. If $\Omega=\mathbb{S}^{d-1}$,
then $\lambda_j=\mu_j$ and the eigenvalues  coincide with $\Lambda_n$ with multiplicity \eqref{mult}.
Both systems $\{\psi_j\}_{j=1}^\infty$ and $\{\omega_j\}_{j=1}^\infty$
are orthonormal in $L_2(\Omega)$.

The following result contains estimates for the Riesz means
of the eigenvalues $\lambda_j$ and $\mu_j$.

\begin{theorem}\label{T:BL}
Let $\sigma\ge1$. Then
inequalities \eqref{Riesz} hold true.
\end{theorem}
\begin{proof} To simplify the notation we shall prove~\eqref{Riesz} for $\mathbb{S}^2$,
the general case is treated in exactly the same way. Furthermore, we start with $\sigma=1$
and first consider the Dirichlet Laplacian.

We extend by zero each eigenfunction $\psi_j$ to the whole $\mathbb{S}^2$ setting
$$
\widetilde\psi_j(s)=
\left\{
  \begin{array}{ll}
    \psi_j(s), & \hbox{$s\in\Omega$;} \\
    0, & \hbox{$s\notin\Omega$.}
  \end{array}
\right.
$$
and expand the result
in spherical harmonics:
$$
\widetilde\psi_j(s)=\sum_{n=0}^\infty\sum_{k=1}^{2n+1}a_{jn}^kY_n^k(s),\quad s\in\mathbb{S}^2,
$$
where
$$
a_{jn}^k=\int_{\mathbb{S}^2}\widetilde\psi_j(s)Y_n^k(s)ds=
\int_{\Omega}\psi_j(s)Y_n^k(s)ds.
$$
Then using \eqref{harmonics} and orthonormality of the $Y_n^k$'s we obtain
$$
\aligned
\sum_{j=1}^\infty(\lambda-\lambda_j)_+=\sum_{j=1}^\infty
\left(\int_\Omega(\lambda+\Delta)\psi_j\psi_jds\right)_+
=\sum_{j=1}^\infty\left(\int_{\mathbb{S}^2}(\lambda+\Delta)\widetilde\psi_j\widetilde\psi_jds\right)_+=\\
=\sum_{j=1}^\infty\left(\int_{\mathbb{S}^2}(\lambda+\Delta)
\sum_{n=0}^\infty\sum_{k=1}^{2n+1}a_{jn}^kY_n^k(s)
\sum_{n'=0}^\infty\sum_{k'=1}^{2n'+1}a_{jn'}^{k'}Y_{n'}^{k'}(s)ds\right)_+=\\
=\sum_{j=1}^\infty\left(
\sum_{n=0}^\infty\bigl(\lambda-n(n+1)\bigr)\sum_{k=1}^{2n+1}(a_{jn}^k)^2\right)_+.
\endaligned
$$
We continue
$$
\aligned
\sum_{j=1}^\infty(\lambda-\lambda_j)_+\le
\sum_{j=1}^\infty
\sum_{n=0}^\infty\bigl(\lambda-n(n+1)\bigr)_+\sum_{k=1}^{2n+1}(a_{jn}^k)^2=\\
=\sum_{n=0}^\infty\bigl(\lambda-n(n+1)\bigr)_+\sum_{k=1}^{2n+1}
\sum_{j=1}^\infty(a_{jn}^k)^2.
\endaligned
$$
We now consider the last double sum
$$
\aligned
\sum_{k=1}^{2n+1}
\sum_{j=1}^\infty(a_{jn}^k)^2=\sum_{k=1}^{2n+1}
\sum_{j=1}^\infty(\psi_j,Y_n^k)^2_{L_2(\Omega)}=
\sum_{k=1}^{2n+1}\int_{\Omega}Y_n^k(s)^2ds=\\
=\int_{\Omega}\sum_{k=1}^{2n+1}Y_n^k(s)^2ds=
\int_{\Omega}\frac{2n+1}{4\pi}ds=(2n+1)\frac{|\Omega|}{4\pi},
\endaligned
$$
where the second equality is the Parceval identity
for the expansion of a fixed function $Y_n^k(s)\vert_{s\in\Omega}$
in the Fourier series with respect to a complete
orthonormal system  $\{\psi_j\}_{j=1}^\infty$ in
$L_2(\Omega)$, while the forth equality is precisely~\eqref{identity}.

As a result, we obtain
\begin{equation}\label{D2}
\sum_{j=1}^\infty(\lambda-\lambda_j)_+\le
\frac{|\Omega|}{4\pi}\sum_{n=0}^\infty\bigl(\lambda-n(n+1)\bigr)_+(2n+1),
\end{equation}
which is the first inequality in \eqref{Riesz} for $\mathbb{S}^2$ and $\sigma=1$.

We now consider the Neumann Laplacian. It is convenient to
denote
$$
\varphi_\lambda(t):=(\lambda-t)_+.
$$

Then
$$
\sum_{j=1}^\infty(\lambda-\mu_j)_+=
\sum_{j=1}^\infty\varphi_\lambda(\mu_j)=
\sum_{j=1}^\infty\varphi_\lambda(\mu_j)\|\omega_j\|^2_{L_2(\Omega)},
$$
where $\omega_j$ are the orthonormal eigenfunctions defined
in~\eqref{DN}. We expand the $\omega_j$'s in the spherical
harmonics:
$$
\omega_j(s)=\sum_{n=0}^\infty\sum_{k=1}^{2n+1}c_{jn}^kY_n^k(s),\ s\in\Omega,
\ \text{where}\
 c_{jn}^k=\int_\Omega\omega_j(s)Y_n^k(s)ds,
$$
so that
$$
1=\|\omega_j\|^2_{L_2(\Omega)}=\sum_{n=0}^\infty\sum_{k=1}^{2n+1}(c_{jn}^k)^2,
$$
and as before
$$
\sum_{k=1}^{2n+1}\sum_{j=1}^\infty (c_{jn}^k)^2=
\int_{\Omega}\sum_{k=1}^{2n+1}Y_n^k(s)^2ds=
(2n+1)\frac{|\Omega|}{4\pi}\,.
$$
Therefore setting
$\sum_{k=1}^{2n+1} (c_{jn}^k)^2=:b_{jn}$, we see that
$$
\sum_{j=1}^\infty b_{jn}=(2n+1)\frac{|\Omega|}{4\pi}\,\quad\text{and}\quad
\sum_{n=0}^\infty b_{jn}=1.
$$
We continue
$$
\aligned
\sum_{j=1}^\infty(\lambda-\mu_j)_+=\sum_{j=1}^\infty\varphi_\lambda(\mu_j)
\sum_{n=0}^\infty b_{jn}=\\=
\sum_{n=0}^\infty
\left(\sum_{j=1}^\infty\varphi_\lambda(\mu_j)b_{jn}\frac{4\pi}{(2n+1)|\Omega|} \right)
\frac{(2n+1)|\Omega|}{4\pi}\,.
\endaligned
$$
We consider the sum with respect to $j$. Since for each fixed $n$
$$
\frac{4\pi}{(2n+1)|\Omega|}\sum_{j=1}^\infty b_{jn}=1,
$$
and the function $\varphi_\lambda$ is convex, Jensen's inequality
gives that
$$
\sum_{j=1}^\infty\varphi_\lambda(\mu_j)b_{jn}\frac{4\pi}{(2n+1)|\Omega|}\ge
\varphi_\lambda\left(\sum_{j=1}^\infty \mu_jb_{jn}\frac{4\pi}{(2n+1)|\Omega|}\right).
$$
On the other hand,
$$
\aligned
\sum_{j=1}^\infty \mu_jb_{jn}=\sum_{j=1}^\infty \mu_j\sum_{k=1}^{2n+1}(c_{jn}^k)^2=
\sum_{j=1}^\infty \mu_j\sum_{k=1}^{2n+1}(\omega_j,Y_n^k)^2_{L_2(\Omega)}=\\=
\sum_{k=1}^{2n+1}\sum_{j=1}^\infty \mu_j(\omega_j,Y_n^k)^2_{L_2(\Omega)}=
n(n+1)\sum_{k=1}^{2n+1}\|Y_n^k\|^2_{L_2(\Omega)}=\\=
n(n+1)\int_\Omega\sum_{k=1}^{2n+1}Y_n^k(s)^2ds=n(n+1)\frac{(2n+1)|\Omega|}{4\pi},
\endaligned
$$
where the last equality is again precisely~\eqref{identity}, and
 the fourth equality follows from the spectral theorem:
$$
\aligned
\sum_{j=1}^\infty\mu_j(\omega_j,Y_n^k)_{L_2(\Omega)}^2=
\left(\sum_{j=1}^\infty\mu_j(\omega_j,Y_n^k)_{L_2(\Omega)}\omega_j,Y_n^k\right)_{L_2(\Omega)}=\\=
(-\Delta Y_n^k,Y_n^k)_{L^2(\Omega)}=n(n+1)\|Y_n^k\|^2_{L_2(\Omega)}.
\endaligned
$$
Therefore
$$
\varphi_\lambda\left(\sum_{j=1}^\infty \mu_jb_{jn}\frac{4\pi}{(2n+1)|\Omega|}\right)=
\varphi_\lambda(n(n+1))=\bigl((\lambda-n(n+1)\bigr)_+,
$$
and we finally obtain
\begin{equation}\label{N2}
\sum_{j=1}^\infty(\lambda-\mu_j)_+\ge
\frac{|\Omega|}{4\pi}\sum_{n=0}^\infty\bigl(\lambda-n(n+1)\bigr)_+(2n+1),
\end{equation}
which is the second inequality in \eqref{Riesz} for $\mathbb{S}^2$ and $\sigma=1$.

To complete the proof it remains to ``lift'' estimates~\eqref{D2},
\eqref{N2} to the powers $\sigma>1$. This is done by using the argument in~\cite{AisLieb}.
For any real number $E$ evaluation of the integral gives the equality
\begin{equation}\label{AisLieb}
E_+^\sigma=\frac1{c_\sigma}\int_0^\infty(E-t)_+t^{\sigma-2}dt,\qquad c_\sigma=B(2,\sigma-1).
\end{equation}
Then in the  Dirichlet case using~\eqref{D2} we have
$$
\aligned
\sum_{j=1}^\infty(\lambda-\lambda_j)_+^\sigma=
\frac1{c_\sigma}\int_0^\infty(\lambda-\lambda_j-t)_+t^{\sigma-2}dt\le\\\le
\frac{|\Omega|}{4\pi}\frac1{c_\sigma}\int_0^\infty\sum_{n=0}^\infty
\bigl(\lambda-t-n(n+1)\bigr)_+(2n+1)t^{\sigma-2}dt=\\=
\frac{|\Omega|}{4\pi}\sum_{n=0}^\infty\bigl(\lambda-n(n+1)\bigr)_+^\sigma (2n+1),
\endaligned
$$
where we used \eqref{AisLieb} twice with $E=\lambda-\lambda_j$ and $E=\lambda-n(n+1)$.

We can do the same in the Neumann case, however, the direct
proof in the case $\sigma=1$ works for $\sigma>1$, since
the function $\varphi_\lambda(t)^\sigma=(\lambda-t)_+^\sigma$ is also convex.
The proof is complete.
\end{proof}

\begin{remark}
{\rm
Inequalities \eqref{Riesz} turn into equalities
for $\Omega=\mathbb{S}^{d-1}$.
}
\end{remark}

We set for $\sigma=1$
\begin{equation}\label{FSd}
F_{\mathbb{S}^{d-1}}(\lambda):=
\frac{|\Omega|}{\sigma_d}\sum_{n=0}^\infty(\lambda-\Lambda_n)_+ k_d(n),
\end{equation}
and give  a more explicit expression for the function
\begin{equation}\label{flam}
f(\lambda):=\sum_{n=0}^\infty(\lambda-\Lambda_n)_+ k_d(n)
\end{equation}
in estimates~\eqref{Riesz} in the case $\sigma=1$.

\begin{lemma}\label{L:piecewise}
The function $f(\lambda)$ is a piecewise linear
function joining the points $[\Lambda_N, f(\Lambda_N)]$
in the plane $(\lambda,f(\lambda))$. For $\lambda=\Lambda_N$,
$N=0,1,\dots$,
\begin{equation}\label{flamN}
f(\Lambda_N)=\sum_{n=0}^{N-1}k_d(n)(\Lambda_N-\Lambda_n).
\end{equation}

\end{lemma}
\begin{proof} The function $f(\lambda)$ is linear on every
interval $\lambda\in [\Lambda_{N-1},\Lambda_n]$. We have
$f(\Lambda_0)=f(0)=0$ and for $\lambda=\Lambda_N$ only the first
$N$ terms with $n=0,\dots,N-1$ in \eqref{flam} are non-zero, which
gives~\eqref{flamN}.
\end{proof}

\setcounter{equation}{0}
\section{ The 2D sphere $\mathbb{S}^2$}

\begin{lemma}\label{L:S2}
On $\mathbb{S}^2$ it holds
\begin{equation}\label{S2flam}
f(\Lambda_N)=\frac{(N(N+1))^2}2=\frac12\Lambda_N^2.
\end{equation}
\end{lemma}
\begin{proof}
This is a direct calculation using~\eqref{flamN} and that $\Lambda_n=n(n+1)$, $k_3(n)=2n+1$.
Alternatively, we may use the general formula~\eqref{genform} which
works for all dimensions.
\end{proof}

\begin{corollary}\label{Cor:geeq}
For all $\lambda\ge0$
\begin{equation}\label{S2flamfull}
F_{\mathbb{S}^2}(\lambda):=\frac{|\Omega|}{4\pi}\sum_{n=0}^\infty\bigl(\lambda-n(n+1)\bigr)_+(2n+1)\ge
\frac 1{8\pi}|\Omega|\lambda^2=\mathrm{L}^\mathrm{cl}_{1,2}|\Omega|\lambda^2,
\end{equation}
while
\begin{equation}\label{S2flamfullequal}
F_{\mathbb{S}^2}(\lambda)=
\frac 1{8\pi}|\Omega|\lambda^2\quad
\text{for}\quad   \lambda=\Lambda_N=N(N+1),\ N=0,1,\dots\,.
\end{equation}
\end{corollary}
\begin{proof} Equality \eqref{S2flamfullequal} is just~\eqref{S2flam}, and
inequality \eqref{S2flamfull} follows by convexity (see the top graph in
Fig.~\ref{fig:GraphsofLegendretrans}).
\end{proof}

Recalling estimate~\eqref{Riesz} for the Neumann eigenvalues $\mu_k$ we
finally obtain the inequality
$$
\sum_{j=1}^\infty(\lambda-\mu_j)_+\ge\frac 1{8\pi}|\Omega|\lambda^2.
$$

We now take the Legendre transform of both sides of this inequality~\cite{L-W}.
We recall that for a convex function $g(x)$ on $\mathbb{R}_+$ the Legendre
transform  $g^\vee(p)$ given by
$$g^\vee(p):=\sup_{x\ge0}(px-g(x)).$$
The Legendre transform of the right-hand side is straightforward.
For the left-hand side we have~\cite{L-W}
\begin{equation}\label{Legendre}
\left(\sum_{j=1}^\infty(\lambda-\mu_j)_+\right)^\vee (p)=
(p-[p])\mu_{[p]+1}+\sum_{k=1}^{[p]}\mu_k.
\end{equation}
If
$g(x)\ge h(x)$, then $g^\vee(p)\le h^\vee(p)$, and setting $p=n\in\mathbb{N}$ we obtain
the following result.
\begin{theorem}\label{T:S2Neumann}
Let $\Omega$ be an open domain on $\mathbb{S}^2$ with two-dimensional
surface area $|\Omega|$  such that the embedding $H^1(\Omega)\hookrightarrow L_2(\Omega)$
is compact. Then the sum of the first $n$ eigenvalues $\mu_k$
of the Neumann Laplacian satisfies the estimate
\begin{equation}\label{LiYauKroger}
\sum_{k=1}^n\mu_k\le\frac{2\pi}{|\Omega|}\,n^2.
\end{equation}
\end{theorem}
\begin{remark}
{\rm
We point out here that while the ordering of the numbers $\mu_j$'s in the left-hand side
in~\eqref{Legendre}
does not play a role, the Legendre transform on the right-hand side
 automatically orders them in the non-decreasing way.
}
\end{remark}

For the lower bound for the Dirichlet eigenvalues we
need an explicit expression  for the Legendre transform $F_{\mathbb{S}^2}^\vee(p)$
of the function $F_{\mathbb{S}^2}(\lambda)$.

Setting
\begin{equation}\label{alpha}
\alpha:=\frac{|\Omega|}{8\pi}
\end{equation}
we first observe that for $\lambda\in[\Lambda_{N-1},\Lambda_N]$
\begin{equation}\label{exprflam}
F_{\mathbb{S}^2}(\lambda)=
\alpha\bigl((\Lambda_{N-1}+\Lambda_N)\lambda-\Lambda_{N-1}\Lambda_N\bigr)
=\alpha\bigl(2N^2\lambda-N^2(N^2-1)\bigr).
\end{equation}
In fact, in view of~\eqref{S2flamfullequal} we only have to verify that
$F_{\mathbb{S}^2}(\lambda)$ in~\eqref{exprflam} satisfies
$F_{\mathbb{S}^2}(\Lambda_{N-1})=\alpha\Lambda_{N-1}^2$ and
$F_{\mathbb{S}^2}(\Lambda_{N})=\alpha\Lambda_{N}^2$.
Next,  the slope
of the straight edge of the graph of $F_{\mathbb{S}^2}(\lambda)$
on   each interval $\lambda\in[\Lambda_{N-1},\Lambda_N]$
is equal to
$$
\frac{\alpha(\Lambda_{N}^2-\Lambda_{N-1}^2)}{\Lambda_{N}-\Lambda_{N-1}}=
\alpha({\Lambda_{N}+\Lambda_{N-1}})
=\alpha\bigl(N(N+1)+(N-1)N\bigr)=2\alpha N^2.
$$

We now find the expression for $F_{\mathbb{S}^2}^\vee(p)$. The
function $F_{\mathbb{S}^2}^\vee(p)$ is a continuous monotone
increasing piecewise linear function whose graph is a polygonal
line with break points located at those points $p$ where the line
$p\lambda$ is parallel to the corresponding straight edge of the
graph of the function $F_{\mathbb{S}^2}(\lambda)$. In other words, when
the slope $p$ is equal to $2\alpha N^2$. See Fig.~\ref{fig:GraphsofLegendretrans}.

\begin{figure}[htb]
\def\svgwidth{13cm}
\def\svgheight{9cm}
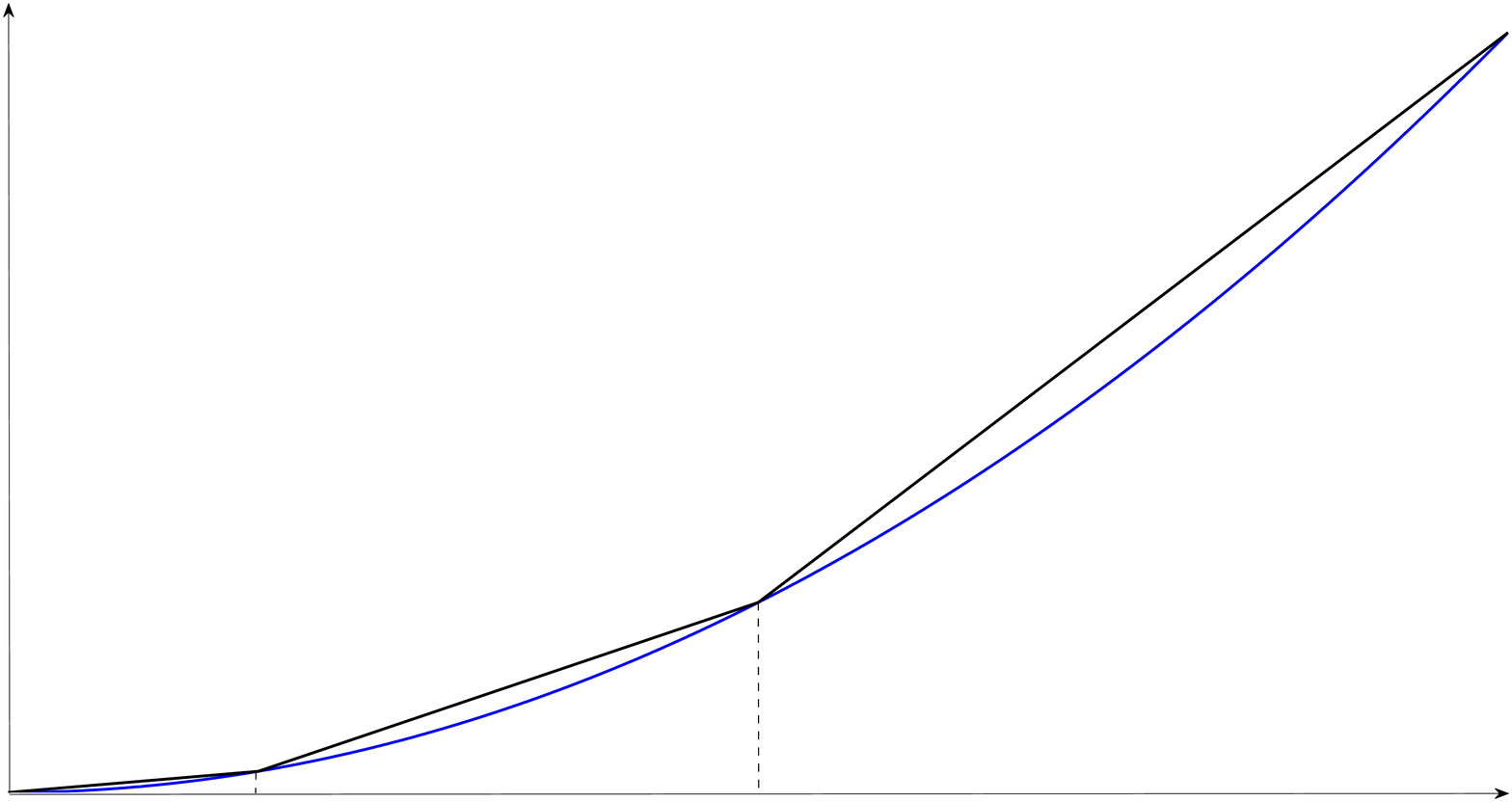
\def\svgwidth{13cm}
\def\svgheight{9cm}
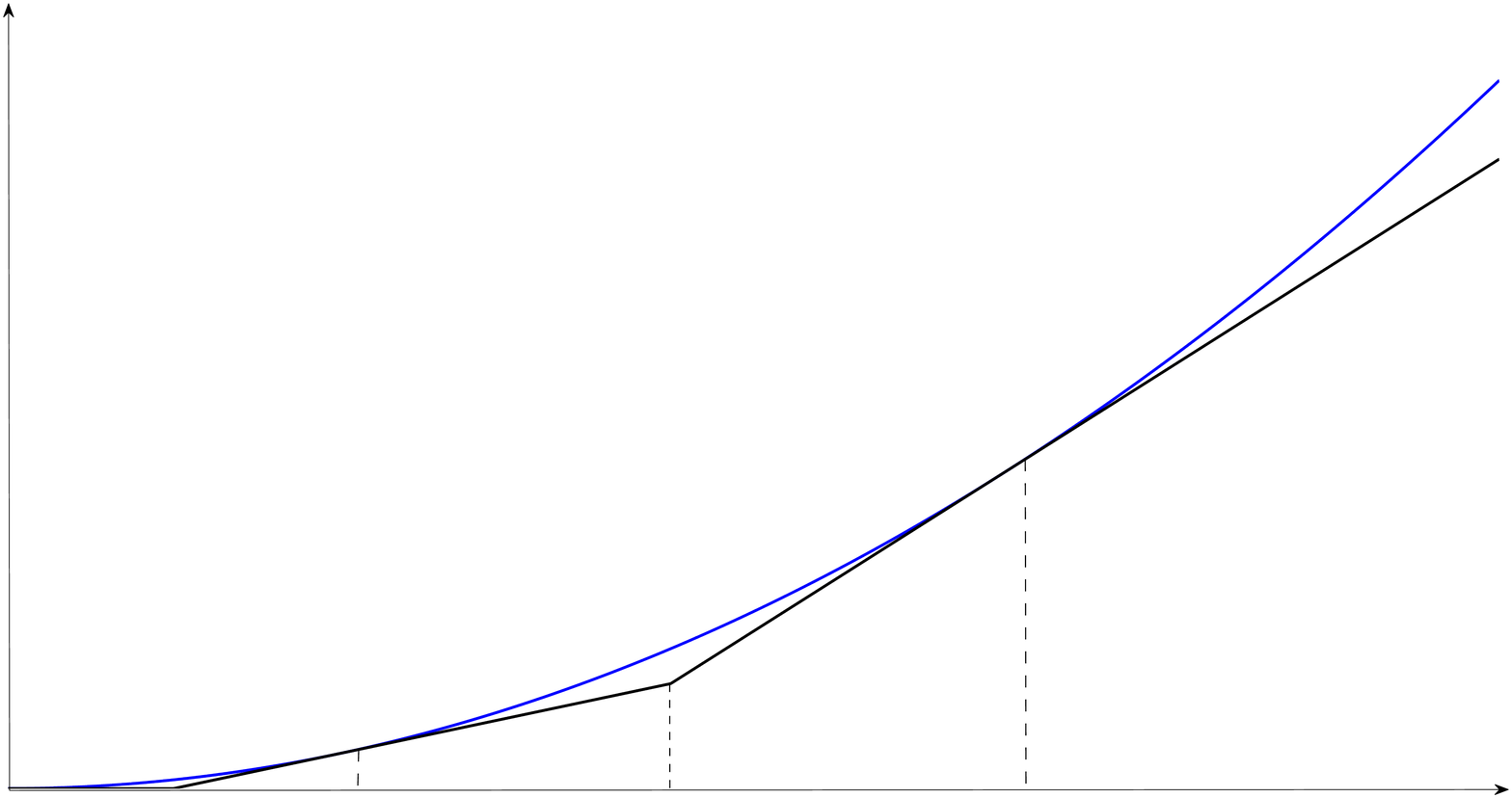
\caption{The graph of  $\lambda\to\alpha\lambda^2$ and the polygonal
graph of $F_{\mathbb{S}^2}(\lambda)$ are shown in the top figure
with break points at $\lambda=\Lambda_1,\Lambda_2\dots\,\,$. The Legendre
transforms $(\alpha\lambda^2)^\vee(p)=p^2/(4\alpha)$ and
$(F_{\mathbb{S}^2})^\vee(p)$ are shown in the bottom figure.
The  graph of $(F_{\mathbb{S}^2})^\vee(p)$
has break points at $p=2\alpha N^2$
and is tangent to the curve $p^2/(4\alpha)$ at $p=2\alpha\Lambda_N$, $N=0,1,\dots\,$.}
\label{fig:GraphsofLegendretrans}
\end{figure}

Having said that we have
$$
\aligned
F_{\mathbb{S}^2}^\vee(p)=\max_{\lambda\in[\Lambda_{N-1},\,\Lambda_N]}
\bigl(p-\alpha(\Lambda_{N-1}+\Lambda_N)\bigr)\lambda+\alpha\Lambda_{N-1}\Lambda_N,\\
p\in\bigl[\alpha(\Lambda_{N-2}+\Lambda_{N-1}),\alpha(\Lambda_{N-1}+\Lambda_{N})\bigr]=
\bigl[2\alpha(N-1)^2,2\alpha N^2\bigr].
\endaligned
$$
The coefficient of $\lambda$ is negative, therefore $\lambda$ has to be the
smallest possible: $\lambda=\Lambda_{N-1}$. We finally obtain
$$
\aligned
F_{\mathbb{S}^2}^\vee(p)&=N(N-1)\bigl(p-\alpha N(N-1)\bigr),\\
 p&\in[2\alpha(N-1)^2,2\alpha N^2],\ N=1,2,\dots\,.
\endaligned
$$
Observe that $F_{\mathbb{S}^2}^\vee(p)=0$ for $p\in[0,2\alpha]$.

We estimate $F_{\mathbb{S}^2}^\vee(p)$ from below. Let
$p\in[2\alpha(N-1)^2,2\alpha N^2]$, that is,
\begin{equation}\label{N1}
\sqrt{\frac p{2\alpha}}\le N\le\sqrt{\frac p{2\alpha}}+1.
\end{equation}
For a fixed $p$ we now look at the function
$F_{\mathbb{S}^2}^\vee(p)=N(N-1)\bigl(p-\alpha N(N-1)\bigr)$, where
$N$ satisfies~\eqref{N1}, as a quadratic parabola $x\to xp-\alpha
x^2$ opening down. Then
\begin{multline*}
F_{\mathbb{S}^2}^\vee(p)\ge
\min\biggl\{N(N-1)\bigl(p-\alpha N(N-1)\bigr)\vert_{N=\sqrt{\frac p{2\alpha}}},\\N(N-1)\bigl(p-\alpha N(N-1)\bigr)\vert_{N=\sqrt{\frac p{2\alpha}}+1}\biggr\}.
\end{multline*}
However, both terms on the right-hand side are equal:
for $u:=\sqrt{p/(2\alpha)}$ and $\alpha=p/(2u^2)$ we have
$$
u(u-1)\left(p-\frac p{2u^2}u(u-1)\right)=
u(u+1)\left(p-\frac p{2u^2}u(u+1)\right)=\frac p2(u^2-1).
$$
Therefore
$$
F_{\mathbb{S}^2}^\vee(p)\ge\frac p2(u^2-1)=\frac1{4\alpha}p(p-2\alpha)=
\frac{2\pi}{|\Omega|}\,p\left(p-\frac{|\Omega|}{4\pi}\right).
$$
Thus, we have proved the following lower bound
for the sums of the Dirichlet eigenvalues.
\begin{theorem}\label{T:DirS2}
Let $\Omega\subseteq\mathbb{S}^2$  be an arbitrary domain. Then for
any $n\ge1$ the sum of the first $n$ Dirichlet eigenvalues
satisfies
\begin{equation}\label{lowerDirS2}
\sum_{k=1}^n\lambda_k\ge\frac{2\pi}{|\Omega|}\,n\left(n-\frac{|\Omega|}{4\pi}\right).
\end{equation}
\end{theorem}

\begin{corollary} Since $n\lambda_n\ge\sum_{k=1}^n\lambda_k$, the lower bound~\eqref{lowerDirS2}
gives that for each $n\ge1$
\begin{equation}\label{indlower}
\lambda_n\ge\frac{2\pi}{|\Omega|}\,\left(n-\frac{|\Omega|}{4\pi}\right).
\end{equation}
\end{corollary}

Setting $n=1$ we obtain
\begin{equation}\label{lambda1}
\lambda_1=\lambda_1(\Omega)\ge\frac{2\pi}{|\Omega|}\,\left(1-\frac{|\Omega|}{4\pi}\right).
\end{equation}
The right-hand side vanishes as $|\Omega|\to4\pi$ which is not
surprising since on the whole sphere the first eigenvalue is zero.

\begin{remark}
{\rm If $\Omega=\mathbb{S}^2$ with $|\Omega|=4\pi$, then  the
eigenvalues are  $\Lambda_n=n(n+1)$ with multiplicities $2n+1$
starting from $n=0$. Since
$$\sum_{n=0}^{N-1}(2n+1)=N^2\quad\text{and}\quad
\sum_{n=0}^{N-1}n(n+1)(2n+1)=\frac12{N^2(N^2-1)},
$$
it follows that inequality~\eqref{lowerDirS2} turns into the
equality for $n=N^2$.
 }
\end{remark}

\begin{remark}
{\rm In view of \eqref{S2flamfullequal}
the points $(\lambda, F_{\mathbb{S}^2}(\lambda))$ with $\lambda=\Lambda_N$
lie on the parabola $\alpha\lambda^2$. Therefore
$F_{\mathbb{S}^2}^\vee(p)=(\alpha\lambda^2)^\vee(p)$ for $p=2\alpha\Lambda_N=2\alpha N(N+1)$.
If $\alpha=1/4$, then $p\in\mathbb{N}$ for all $N=1,2,\dots\,$.
Therefore we have shown that for  $|\Omega|=2\pi=|\mathbb{S}^2|/2$  and  $n=N(N+1)/2$
 the sum of the first $n$ Dirichlet eigenvalues  satisfies the lower bound
$$
\sum_{k=1}^n\lambda_k\ge n^2, \quad\text{where}\ n={N(N+1)}/2,\ N=1,2,\dots\,.
$$
In particular, $\lambda_1\ge1$, while the universal estimate
\eqref{lambda1} gives $\lambda_1\ge1/2$.
}
\end{remark}

\subsection*{The vector case}
In the vector case we first define the Laplace operator acting on
(tangent) vector fields on $\mathbb{S}^2$ as the Laplace--de Rham
operator $-d\delta-\delta d$ identifying $1$-forms and vectors.
Then for a two-dimensional manifold (not necessarily
$\mathbb{S}^2$) we have \cite{I93}
\begin{equation}\label{vecLap}
\mathbf{\Delta} u=\nabla\div u-\rot\rot u,
\end{equation}
where the operators $\nabla=\grad$ and $\div$ have the conventional
meaning. The operator $\rot$ of a vector $u$ is a scalar  and for a
scalar $\psi$, $\rot\psi$ is a vector:
$$
\rot u:=-\div(\mathrm{n}\times u)=\div u^\perp,\qquad
\rot\psi:=-\mathrm{n}\times\nabla\psi=\nabla^\perp\psi,
$$
where $\mathrm{n}$ is the unit outward normal vector, so that in the local
frame
$$-\mathrm{n}\times u=(u_2,-u_1)=:u^\perp.
$$
 We note that for a scalar $\psi$ it
holds
\begin{equation}\label{rotrot}
\rot\rot\psi=-\Delta\psi\ (=-\div\grad\psi).
\end{equation}
 Integrating by parts, that is, using
$$
(\nabla\psi,u)_{L_2(T\mathbb{S}^2)}=-(\psi,\div
u)_{L_2(\mathbb{S}^2)},\quad
(\rot\psi,u)_{L_2(T\mathbb{S}^2)}=(\psi,\rot
u)_{L_2(\mathbb{S}^2)},
$$
we obtain
\begin{equation}\label{byparts}
(-\mathbf{\Delta} u,u)_{L_2(T\mathbb{S}^2)}=\|\rot u\|^2+\|\div u\|^2.
\end{equation}

The vector Laplacian has a complete in $L_2(T\mathbb{S}^2)$
orthonormal basis of vector-valued eigenfunctions. Using the
notation
\begin{equation}\label{not}
\{y_i\}_{i=1}^\infty=\{
Y_n^1,\dots,Y_n^{2n+1}\}_{n=1}^\infty,\quad
\{\lambda_i\}_{i=1}^\infty=
\underset{\!2n+1\ \text{times}}{\{\Lambda_n,
\dots,\Lambda_n\}_{n=1}^{\infty}}.
\end{equation}
 we have
\begin{equation}\label{bas-vec}
\aligned
-\mathbf{\Delta} w_j=\lambda_j w_j,\qquad -\mathbf{\Delta}
v_j=\lambda_j v_j,
\endaligned
\end{equation}
where
$$
 w_j=\lambda_j^{-1/2}\nabla^\perp y_j,\ \ \div w_j=0,\qquad
v_j=\lambda_j^{-1/2}\nabla y_j,\ \ \rot v_j=0.
$$

Both~(\ref{bas-vec}), and the orthonormality of the $w_j$'s and
$v_j$'s follow from~(\ref{rotrot}), and~(\ref{rotrot}) also implies
the following commutation relations
$$
\mathbf{\Delta}\nabla= \nabla\Delta, \qquad
\mathbf{\Delta}\nabla^\perp=\nabla^\perp\Delta,
$$
which proves~\eqref{bas-vec}. For example, for the second relation
we have
\begin{equation}\label{commut}
\aligned
\mathbf{\Delta} \nabla^\perp \psi=-\rot\rot\nabla^\perp \psi=
-\rot\rot\rot \psi=\rot\Delta \psi=\nabla^\perp \Delta\psi.
\endaligned
\end{equation}
We also point out that the fact that we are dealing with the sphere
$\mathbb{S}^2$ does not play a role and \eqref{vecLap} --
\eqref{commut} hold for any 2D manifold $M$.

Hence,  on $\mathbb{S}^2$, corresponding to the eigenvalue
$\Lambda_n=n(n+1)$, where $n=1,2,\dots$, there are two families of
$2n+1$ orthonormal vector-valued eigenfunctions  $w_n^k(s)$ and $v_n^k(s)$, where
$k=1,\dots,2n+1$ and~(\ref{identity-vec}) gives the following
important identities: for any $s\in\mathbb{S}^2$
\begin{equation}\label{id-vec}
\sum_{k=1}^{2n+1}|w_n^k(s)|^2=\frac{2n+1}{4\pi},\qquad
\sum_{k=1}^{2n+1}|v_n^k(s)|^2=\frac{2n+1}{4\pi}.
\end{equation}
We finally observe that since the sphere is simply connected,
it follows that
$$
\left\{\div u=0,\ \rot u=0\right\}\Rightarrow u=0,
$$
and therefore $-\mathbf{\Delta}$ is strictly positive $-\mathbf{\Delta}\ge
\Lambda_1I=2I$. This fact explains why the Li--Yau bounds below look
exactly the same as in the case of a bounded domain in $\mathbb{R}^2$.

We now consider the Dirichlet eigenvalues in the vector case:
$$
-\mathbf{\Delta}u_j=\lambda_ju_j,
\qquad u_j\vert_{\partial\Omega}=0,
$$
where the vector-valued eigenfunctions $u_j$, $j=1,\dots$ make up a
complete orthonormal family in $L_2(\Omega,T\mathbb{S}^2)$.
\begin{theorem}\label{T:Riesz-vec}
Let $\Omega\subseteq\mathbb{S}^2$  be an arbitrary domain. Then for
$\lambda\ge0$
\begin{equation}\label{D2-vec}
\sum_{j=1}^\infty(\lambda-\lambda_j)_+\le
2\cdot\frac{|\Omega|}{4\pi}\sum_{n=1}^\infty\bigl(\lambda-n(n+1)\bigr)_+(2n+1).
\end{equation}
\end{theorem}
\begin{proof}
Before we proceed with the proof we point out the factor $2$
on the right-hand side and the fact that the summation starts with
$n=1$.

The proof in turn repeats that of Theorem~\ref{T:BL} and therefore will
only be outlined.
Let $\widetilde u_j$ be the extension by zero of $u_j$.
We expand each $\widetilde u_j$ with respect to the orthonormal
basis \eqref{bas-vec}:
$$
\widetilde u_j(s)=\sum_{n=1}^\infty\sum_{k=1}^{2n+1}
\left(a_{jn}^kw_n^k(s)+c_{jn}^kv_n^k(s)\right),\quad s\in\mathbb{S}^2,
$$
where
$$
a_{jn}^k=\bigl(\widetilde u_j(s),w_n^k\bigr)_{L_2(\Omega,T\mathbb{S}^2)},
\qquad
c_{jn}^k=\bigl(\widetilde u_j(s),v_n^k\bigr)_{L_2(\Omega,T\mathbb{S}^2)}.
$$
Then as in the proof of the Dirichlet case in Theorem~\ref{T:BL} we obtain
$$
\sum_{j=1}^\infty(\lambda-\lambda_j)_+\le
\sum_{n=1}^\infty\bigl(\lambda-n(n+1)\bigr)_+\sum_{k=1}^{2n+1}
\sum_{j=1}^\infty\left((a_{jn}^k)^2+(c_{jn}^k)^2\right).
$$

For the double sum  we use the
 identity  \eqref{id-vec} and find that
$$
\aligned
\sum_{k=1}^{2n+1}
\sum_{j=1}^\infty\bigl((a_{jn}^k)^2+(c_{jn}^k)^2\bigr)=\sum_{k=1}^{2n+1}
\sum_{j=1}^\infty\bigl((u_j,w_n^k)^2_{L_2(\Omega,T\mathbb{S}^2)}+
(u_j,v_n^k)^2_{L_2(\Omega,T\mathbb{S}^2)}\bigr)=\\=
\sum_{k=1}^{2n+1}\int_{\Omega}\bigl(|w_n^k(s)|^2+|v_n^k(s)|^2\bigr)ds=
2(2n+1)\frac{|\Omega|}{4\pi}.
\endaligned
$$
 \end{proof}

The following lemma shows that the omission of the zeroth term in
the sum in~\eqref{D2-vec} reverses the inequality
in~\eqref{S2flamfull}, see Fig.~\ref{fig:below}.
\begin{lemma}\label{L:asinR2}
For all $\lambda\ge0$
\begin{equation}\label{asinR2}
F'_{\mathbb{S}^2}(\lambda):=\frac{|\Omega|}{4\pi}\sum_{n=1}^\infty\bigl(\lambda-n(n+1)\bigr)_+(2n+1)
\le\frac 1{8\pi}|\Omega|\lambda^2=\mathrm{L}^\mathrm{cl}_{1,2}|\Omega|\lambda^2.
\end{equation}
\end{lemma}
\begin{proof}
Recalling definition \eqref{S2flamfull} of the function
$F_{\mathbb{S}^2}(\lambda)$ and the explicit expression for it \eqref{alpha}, \eqref{exprflam},
what we have to prove is the following inequality
$$
\alpha\bigl(2N^2\lambda-N^2(N^2-1)\bigr)-2\alpha\lambda\le\alpha\lambda^2,
\qquad\text{for}\quad \lambda\in[\Lambda_{N-1},\Lambda_N]
$$
or
$
\lambda^2-2(N^2-1)\lambda+N^2(N^2-1)\ge0.
$
However, this quadratic inequality holds for all $\lambda$, since its
discriminant  for $N\ge1$ is negative:
$$
-4(N^2-1)\le0.
$$

\end{proof}

\begin{figure}[htb]
\def\svgwidth{10cm}
\def\svgheight{9cm}
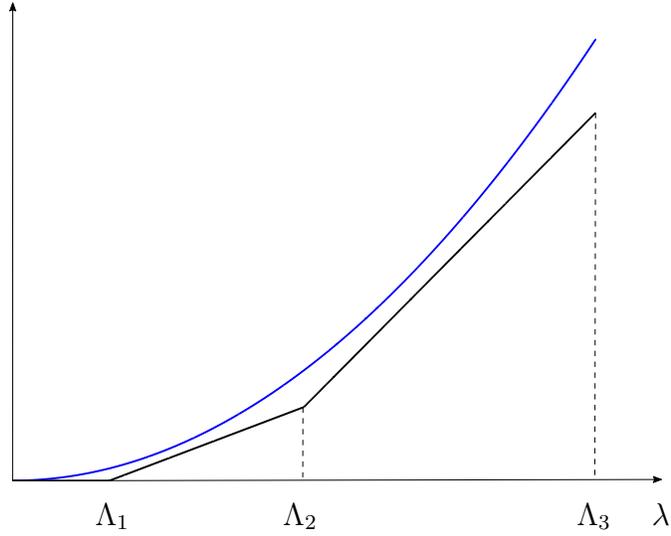
\caption{The polygonal graph of $F'_{\mathbb{S}^2}(\lambda)=F_{\mathbb{S}^2}(\lambda)-2\alpha\lambda$
lies below the parabola $\lambda\to\alpha\lambda^2$.}
\label{fig:below}
\end{figure}

Combining \eqref{D2-vec} and \eqref{asinR2} we obtain that
\begin{equation}\label{D2-vec-class}
\sum_{j=1}^\infty(\lambda-\lambda_j)_+\le
2\mathrm{L}^\mathrm{cl}_{1,2}|\Omega|\lambda^2.
\end{equation}

\begin{theorem}\label{T:DirS2-vec}
Let $\Omega\subseteq\mathbb{S}^2$  be an arbitrary domain. Then for
any $n\ge1$ the sum of the first $n$ Dirichlet eigenvalues
of the vector Laplacian
satisfies a Li--Yau-type lower bound
\begin{equation}\label{lowerDirS2-vec}
\sum_{k=1}^n\lambda_k\ge\frac{\pi}{|\Omega|}\,n^2.
\end{equation}
\end{theorem}
\begin{proof}
Taking the Legendre transform of both sides of~\eqref{D2-vec-class}
we obtain~\eqref{lowerDirS2-vec}.
\end{proof}

We finally consider the case of the Stokes operator in a domain
on $\mathbb{S}^2$. Let $\nu_j$ and $u_j$ be the eigenvalues and the
divergence free
vector-valued eigenfunctions of the Stokes operator
defined by the quadratic form
$$
u\to\int_{\Omega}(\rot u)^2dS,
$$
with domain
$$
u\in H^1_0(\Omega,T\mathbb{S}^2),\quad  \div u=0.
$$
If the boundary $\partial\Omega$ is sufficiently smooth, then
this eigenvalue problem can be witten in the form~\eqref{Stokessmooth}.

\begin{theorem}\label{T:StokesS2}
The sum of the first $n$  eigenvalues
of the Stokes operator on a domain $\Omega\subseteq\mathbb{S}^2$
satisfies a Li--Yau-type lower bound
\begin{equation}\label{StokesS2}
\sum_{k=1}^n\nu_k\ge\frac{2\pi}{|\Omega|}\,n^2.
\end{equation}
\end{theorem}
\begin{proof}
Arguing as in the proof of Theorem~\ref{T:Riesz-vec} but this time
using only the family of divergence free vector-valued eigenfunctions $w_n^k$
 (see \eqref{bas-vec}, \eqref{id-vec})
we obtain
$$
\sum_{j=1}^\infty(\lambda-\nu_j)_+\le
\frac{|\Omega|}{4\pi}\sum_{n=1}^\infty\bigl(\lambda-n(n+1)\bigr)_+(2n+1).
$$
It remains to use~\eqref{asinR2} and apply the Legendre transform.
\end{proof}

\setcounter{equation}{0}
\section{  Neumann problem on domains on  higher dimensional spheres}\label{sec4}

\begin{proposition}\label{P:Sd-1}
Let $d\ge4$. For $\lambda>0$ the function $F_{\mathbb{S}^{d-1}}(\lambda)$
defined in~\eqref{FSd} satisfies the inequality
\begin{equation}\label{Sd-1andRd-1}
F_{\mathbb{S}^{d-1}}(\lambda)>\mathrm{L}^\mathrm{cl}_{1,d-1}|\Omega|\lambda^{1+(d-1)/2}.
\end{equation}
\end{proposition}
\begin{proof}
The proof essentially relies on  the
following general formula established in the Appendix:
\begin{equation}\label{flamd}
\aligned
f(\Lambda_N):=\sum_{n=0}^{N-1}k_d(n)(\Lambda_N-\Lambda_n)
=\frac{(2N+d-1)(2N+d-3)}{d+1}\binom{N+d-2}{d-1}.
\endaligned
\end{equation}
In view of the convexity of both functions in~\eqref{Sd-1andRd-1}
and the polygonal shape of the graph of $F_{\mathbb{S}^{d-1}}(\lambda)$
it suffices to show that \eqref{Sd-1andRd-1} holds for
$\lambda=\Lambda_N=N(N+d-2)$ for all $N=1,2,\dots\,$. In other
words, we have to show that
$$
\frac{(2N+d-1)(2N+d-3)}{\sigma_{d}(d+1)}\binom{N+d-2}{d-1}>
\mathrm{L}^\mathrm{cl}_{1,d-1}\bigl(N(N+d-2)\bigr)^{(d+1)/2}.
$$
The coefficient of the leading term  $N^{d+1}$ in the left-hand side
is
$$
\frac4{\sigma_d(d+1)\Gamma(d)}=\mathrm{L}^\mathrm{cl}_{1,d-1},
$$
where the last  equality follows from~\eqref{class} and the
duplication formula for the gamma function.
Hence  what we have to prove is the
following inequality:
\begin{multline}\label{ineqLambdaN}
(N+(d-1)/2)(N+(d-3)/2)(N+d-2)(N+d-3)\cdots(N+1)N>\\>[N(N+d-2)]^{(d+1)/2}.
\end{multline}

As we have seen, in the two-dimensional case $d=3$ we have equality here,
but we are now dealing with the case $d\ge4$.

The left-hand side contains $d+1$ factors.
We set
$$
N(N+d-2)=\lambda,\qquad N=\sqrt{\lambda+(d-2)^2/4}-(d-2)/2.
$$
Assume first that $d+1$ is even. We combine the factors
in~\eqref{ineqLambdaN} as follows (the first two terms are repeated
twice: first as they  are and the second time in the middle of the
next group of factors)
$$
\aligned
&N(N+d-2)=\lambda,\\
&(N+1)(N+d-3)=\lambda+d-3=:\lambda+\alpha_1,\\
&(N+2)(N+d-4)=\lambda+2(d-4)=:\lambda+\alpha_2,\\
&\vdots\\
&(N+(d-3)/2)(N+(d-1)/2)=\lambda+(d-1)(d-3)/4=:\lambda+\alpha_{(d-1)/2},\\
&(N+(d-1)/2)(N+(d-3)/2)=\lambda+(d-1)(d-3)/4=:\lambda+\alpha_{(d-1)/2},
\endaligned
$$
where all $\alpha_j>0$.
Now the inequality we are looking for becomes obvious
$$
\lambda(\lambda+\alpha_{(d-1)/2})\prod_{j=1}^{(d-3)/2}(\lambda+\alpha_{j})>\lambda^{(d+1)/2}.
$$

Finally, if $d+1$ is odd, we act similarly, and the factor without pair is
$N+(d-2)/2$. But $N+(d-2)/2=\sqrt{\lambda+(d-2)^2/4}$ and we obtain instead
$$
\lambda\sqrt{\lambda+(d-2)^2/4}\prod_{j=1}^{(d-2)/2}(\lambda+\alpha_{j})>\lambda^{(d+1)/2}.
$$
\end{proof}

\begin{theorem}\label{Corr:SdandRd}
For any domain $\Omega\subseteq\mathbb{S}^{d-1}$
for which the embedding $H^1(\Omega)\hookrightarrow L_2(\Omega)$ is compact
the following lower bound holds for the Riesz means of order $1$ of
the Neumann eigenvalues
\begin{equation}\label{asinRd}
\sum_{j=1}^\infty(\lambda-\mu_j)_+\ge
\mathrm{L}^\mathrm{cl}_{1,d-1}|\Omega|\lambda^{1+(d-1)/2}.
\end{equation}
\end{theorem}
\begin{proof} The proof is a combination of~\eqref{Riesz}, \eqref{FSd} and~\eqref{Sd-1andRd-1}.
\end{proof}

\begin{theorem}\label{T:S3Neumann}
The sum of the first $n$ eigenvalues $\mu_k$
of the Neumann Laplacian on a domain $\Omega\subseteq\mathbb{S}^{d-1}$,
$d-1\ge2$,  satisfies the estimate
\begin{equation}\label{LiYauKroger3}
\sum_{k=1}^n\mu_k\le
\frac {d-1}{d+1}\left(\frac{(2\pi)^{d-1}}{\omega_{d-1}|\Omega|}\right)^{2/(d-1)}\,n^{1+2/(d-1)},
\end{equation}
while each eigenvalue satisfies for $k=0,1,\dots$  the upper bound
\begin{equation}\label{Krog_Lap}
\mu_{k+1}\le\left(\frac{d+1}{2}\right)^{2/(d-1)}
\left(\frac{(2\pi)^{d-1}}{\omega_{d-1}|\Omega|}\right)^{2/(d-1)}\,k^{2/(d-1)}.
\end{equation}
\end{theorem}
\begin{proof}
Taking the Legendre transform of~\eqref{asinRd}  we obtain~\eqref{LiYauKroger3}.
Next, we prove \eqref{Krog_Lap} following \cite[Theorem~3.2]{LapFA}: since $\mu_1=0$,
the counting function $N(\lambda,-\Delta_\Omega^N):=\#\{k,\mu_k<\lambda\}$
satisfies
$$
N(\lambda,-\Delta_\Omega^N)\ge\frac1\lambda
\sum_{j=1}^\infty(\lambda-\mu_j)_+.
$$
Therefore, in view of \eqref{asinRd},
$$
N(\lambda,-\Delta_\Omega^N)\ge
\mathrm{L}^\mathrm{cl}_{1,d-1}|\Omega|\lambda^{(d-1)/2}=
\frac2{d+1}\mathrm{L}^\mathrm{cl}_{0,d-1}|\Omega|\lambda^{(d-1)/2},
$$
which is equivalent to \eqref{Krog_Lap}.
\end{proof}
\begin{remark}
{\rm
Upper bounds \eqref{LiYauKroger3} and \eqref{Krog_Lap} are exactly
the same as in the case $\Omega\subset\mathbb{R}^{d-1}$, see~\cite{Kroger,LapFA}.
 }
\end{remark}

\setcounter{equation}{0}
\section{Appendix. Calculation of $f(\Lambda_N)$}\label{secApp}

We first recall the formula for the summation of the
multiplicities~\eqref{mult}
\begin{equation}\label{summult}
\sum_{n=0}^Nk_d(n)=k_{d+1}(N).
\end{equation}
In fact, since
$$
k_d(n)=\binom{d+n-1}{n}-\binom{d+n-3}{n-2},
$$
(where the right hand side is the difference between the dimensions
of the homogeneous polynomials of degrees $n$ and $n-2$ in
$\mathbb{R}^d$) the sum telescopes and we obtain~\eqref{summult}.

We now consider
\begin{equation}\label{flamN1}
f(\Lambda_N)=\sum_{n=0}^{N-1}k_d(n)(\Lambda_N-\Lambda_n)
=\Sigma_1(d,N)-\Sigma_2(d,N),
\end{equation}
where
$$
\Sigma_1(d,N)=\Lambda_N\sum_{n=0}^{N-1}k_d(n),\qquad
\Sigma_2(d,N)=\sum_{n=0}^{N-1}k_d(n)\Lambda_n.
$$

For the first sum we use \eqref{summult} and find that
\begin{equation}\label{first}
\Sigma_1(d,N)=N(N+d-2)k_{d+1}(N-1)=
N(2N+d-3)\binom{N+d-2}{d-1}.
\end{equation}

Using \eqref{mult} we write the second sum as follows
$$
\aligned
\Sigma_2(d,N)=\frac1{d-2}\sum_{n=0}^{N-1}\binom{n+d-3}{d-3}
(2n+d-2)n(n+d-2)=\\=
\frac1{d-2}\sum_{n=0}^{N-1}\binom{n+d-3}{d-3}
\left(2n^3+3(d-2)n^2+(d-2)^2n\right).
\endaligned
$$
The last factor is a polynomial with respect to $n$ of degree three
without a constant term, and we represent it in the basis
$$
\left\{\frac n{d-2},\ \frac{n(n-1)}{(d-2)(d-1)},\ \frac{n(n-1)(n-2)}{(d-2)(d-1)d}\right\}
$$
as follows
$$
\aligned
&2n^3+3(d-2)n^2+(d-2)^2n=\\=&(d-2)(d-1)d
\left(2\frac{n(n-1)(n-2)}{(d-2)(d-1)d}+
3\frac{n(n-1)}{(d-2)(d-1)}+
\frac{n}{(d-2)}\right).
\endaligned
$$
Therefore
$$
\aligned
\Sigma_2(d,N)=(d-1)d\sum_{n=0}^{N-1}\biggl[2\binom{n+d-3}{d-3}\frac{n(n-1)(n-2)}{(d-2)(d-1)d}+\\
+3\binom{n+d-3}{d-3}\frac{n(n-1)}{(d-2)(d-1)}+
\binom{n+d-3}{d-3}\frac{n}{d-2}\biggr]=\\=
(d-1)d\sum_{n=0}^{N-1}\biggl[2\binom{n+d-3}{d}+
3\binom{n+d-3}{d-1}+\binom{n+d-3}{d-2}\biggr].
\endaligned
$$
By the well-known property that
$\binom{m}{k}=\binom{m+1}{k+1}-\binom{m}{k+1}$ the three sums
telescope and we obtain
$$
\Sigma_2(d,N)=(d-1)d\biggl[2\binom{N+d-3}{d+1}+
3\binom{N+d-3}{d}+\binom{N+d-3}{d-1}\biggr].
$$

In view of \eqref{first} we single out the factor
$\binom{N+d-2}{d-1}$ here and after straight forward calculations
we obtain
$$
\Sigma_2(d,N)=\binom{N+d-2}{d-1}\frac{(d-1)(N-1)}{N+d-2}
\biggl[\frac{2(N-2)(N-3)}{d+1}+3(N-2)+d\biggr].
$$
However, the last factor equals
$$
\frac{(N+d-2)(2N+d-3)}{d+1},
$$
which  gives
\begin{equation}\label{second}
\Sigma_2(d,N)=\frac{d-1}{d+1}\binom{N+d-2}{d-1}(N-1)(2N+d-3).
\end{equation}
We finally obtain
\begin{equation}\label{genform}
f(\Lambda_N)
=\Sigma_1(d,N)-\Sigma_2(d,N)=
\frac{(2N+d-1)(2N+d-3)}{d+1}\binom{N+d-2}{d-1},
\end{equation}
which is \eqref{flamd}.

\subsection*{Acknowledgements}\label{SS:Acknow}
A.I.  acknowledges the warm hospitality of the Mit\-tag--Leffler
Institute where this work has started. We also thank D.N.\,Tulya\-kov
for his help in the proof of the  summation formula in the
Appendix. The work of A.I. was supported by the Russian Science
Foundation (grant no. 14-21-00025). A.L. was partially funded by the grant of the Russian Federation Government to support research under the supervision of a leading scientist at the Siberian Federal University, 14.Y26.31.0006.

\end{document}